\documentclass[smallextended]{svjour3}       % onecolumn (second format)
\smartqed  % flush right qed marks, e.g. at end of proof
\usepackage{graphicx}
\usepackage{mathptmx}      % use Times fonts if available on your TeX system
\usepackage{amsmath, amsfonts, amssymb}
\usepackage{enumerate}
\begin{document}

\title{Extension of order bounded operators}

%\titlerunning{Characterizations of Riesz spaces in the sense of S.O.B. operators}        % if too long for running head

\author{Kazem Haghnejad Azar}

\authorrunning{K. Haghnezhad Azar\MakeLowercase{}} % if too long for running head
\institute{K. Haghnezhad Azar \at Department  of  Mathematics  and  Applications, Faculty of Basic  Sciences, University of Mohaghegh Ardabili, Ardabil, Iran.\\
Email: haghnejad@uma.ac.ir}

\date{Received: date / Accepted: date}

\maketitle

\begin{abstract}
Assume that a  normed lattice  $E$  is order dense majorizing of a vector lattice $E^t$.  There is an  extension norm $\Vert.\Vert_t$ for $E^t$  and we extend some lattice and topological  properties of normed lattice $(E,\Vert.\Vert)$ to new  normed lattice $(E^t,\Vert.\Vert_t$).  
 For a Dedekind complete Banach lattice $F$, $T^t$ is an extension of  $T$ from $E^t$ into $F$ whenever $T$ is an order bounded operator from $E$ into  $F$.  For each positive operator $T$, we have  $\Vert T\Vert=\Vert T^t\Vert$ and  we show that
 $T^t$ is a lattice homomorphism from $E^t$ into $F$  and  moreover   $T^t\in \mathcal{L}_n(E^t,F)$ whenever  $0\leq T\in \mathcal{L}_n(E,F)$ and $T(x\wedge y)=Tx \wedge Ty$ for each $0\leq x,y\in E$.
  We also extend some lattice and topological properties of $T\in \mathcal{L}_b(E,F)$  to the extension operator $T^t\in \mathcal{L}_b(E^t,F)$.
%%%%%%%%%%%%%%%%%%%%%%%%%%%%
\end{abstract}

\section{Introduction and Preliminaries}
In the  section 1.2 of \cite{2},  authors have been studied  extension  operators on vector latticess. In this paper, we will study this problem in different way and we extend some results to general case.
Assume that  $E$ is a normed lattice and sublattice of $G$, and $E$ is  order dense majorizing of a vector lattice $E^t\subseteq G$. The aim of this manuscript are in the following:
\begin{enumerate}
\item 	We extend the norm from $E$ to $E^t$.
\item 	Assume that $T$ is an order bounded operator from $E$ into Dedekind complete normed lattice $F$.  $T^t$ is a linear extension of $T$, from $E^t$ into $F$, in the sense that if $S:E^t\rightarrow F$ is any operator that extends $T$ by same way, then $T^t=S$.
\item	We also extend some lattice and topological properties from $E$ and $T$ for $E^t$ and $T^t$, respectively.
\end{enumerate}

To state our result, we need to fix some notation and recall some definitions.
A Banach lattice $E$ has order continuous norm if $\| x_\alpha\|\rightarrow 0$ for every decreasing net $(x_\alpha)_\alpha$ with $\inf_\alpha x_\alpha=0$.
Let $E$, $F$ be Riesz spaces. An operator $T:E\rightarrow F$ is said to be order bounded if it maps each order bounded subset of $E$ into order bounded subset of $F$. The collection of all order bounded operators from a Riesz space $E$ into a Riesz space $F$ will be denoted by $\mathcal{L}_b(E,F)$.  A linear operator between two  Riesz spaces is order continuous (resp. $\sigma$-order continuous) if it maps order null nets (resp. sequences) to order null nets (resp. sequences). The collection of all order continuous (resp. $\sigma$-order continuous) linear operators from vector lattice $E$ into vector lattice $F$ will be denoted by $\mathcal{L}_n(E,F)$ (resp. $\mathcal{L}_c(E,F)$). Let us say that a vector subspace $G$ of an ordered vector space $E$ is majorizing $E$ whenever for each $x\in E$ there exists some $y\in G$ with $x\leq y$. A vector sublattice $G$ of vector lattice $E$ is said to be order dense in $E$ whenever for each $0<x\in E$ there exists some $y\in G$ with $0<y\leq x$. A Dedekind complete vector lattice $E$ is said to be a Dedekind completion of the vector lattice $G$ whenever $E$ is lattice isomorphism to a majorizing order dense sublattice of $E$. A subset $A$ of a vector lattice $E$ is said to be order closed if it follows from $\{x_\alpha\}\subseteq A$ and $x_\alpha\xrightarrow{o} x$ in $E$ that $x\in A$.
A lattice norm $\Vert .\Vert$ on a vector lattice $E$ is said to be a Fatou norm (or that  $\Vert .\Vert$ satisfies the Fatou property) if $0\leq x_\alpha\uparrow x$ in $E$ implies $\Vert x_\alpha \Vert\uparrow \Vert\ x\Vert$.
 $\sigma$-Fatou norm has similar definition. An operator $T:E\rightarrow E$ on a vector lattice is said to be band preserving whenever $T$ leaves all bands of $E$ invariant, i.e., whenever $T(B)\subseteq B$ holds for each band $B$ of $E$. An operator $T:E\rightarrow F$ between two vector lattices is said to preserve disjointness whenever $x\perp y$ in $E$ implies $Tx\perp Ty$ in $F$.
For a normed lattice $E$, $E^\prime$ is the its order dual and $\sigma (E,E^\prime )$ is the weak topology for $E$.
For unexplained terminology and facts on  Banach lattices and positive operators, we refer the reader to \cite{1,2}.

\section{Extension of norm to an order completion}
Let $E$ be an Archimedean vector lattice. Then there exists a Dedekind complete vector lattice $E^\delta$ contains a majorizing order dense vector subspace that is Riesz isomorphism to $E$, which we identify with $E$. 
 $E^\delta$ is called the Dedekind completion of $E$.  
Throughout this  manuscript,  we will assume that the vector lattices under consideration are Archimedean.
   $E$ and $G$ are normed lattice  and  vector lattice, respectively, in which  $E$ is order dense  majorizing of $G$. The universal completion of a vector lattice $E$ will be denoted by $E^u$. By Theorem 7.21 \cite{1}, every  vector lattice has a unique universal completion. 
\begin{theorem}\label{t2.1}
For each $x\in G$, $\rho(x)=\sup\{\Vert z\Vert:~z\leq \vert x\vert ,~z\in E^+\}$ is a norm for $G$ and moreover $(G, \rho(x))$ is a normed lattice. 
\end{theorem}

\begin{proof}
It is clear that $\rho(x)=0$ if and only if $x=0$, and $\rho(\lambda x)=\vert\lambda\vert \rho(x)$ for each real number $\lambda$ and $x\in G$. Now we prove that $\rho( x+y)\leq \rho( x)+\rho( y)$ whenever $x,y\in E^t$. \\
Let $x, y\in E^t$. Fix  $z\in E^+$ such that $z\leq \vert x+y\vert$ . By Riesz Decomposition property, Theorem 1.10 \cite{1}, there are $z_1,z_2\in G$ such that  $\vert z_1\vert\leq \vert x\vert $,  $\vert z_2\vert\leq \vert y\vert $ and   $z=z_1+z_2$. Since $E$ is order dense in $G$, there are $w_1,w_2\in E^+$ such that   $\vert z_1\vert\leq w_1\leq\vert x\vert $ and  $\vert z_2\vert\leq w_2\leq \vert y\vert $.  It follows that 
\begin{align*}
z=\vert z_1+z_2\vert\leq \vert z_1\vert+\vert z_2\vert\leq w_1+w_2\leq \vert x\vert+\vert y\vert.
\end{align*}
Then we have
\begin{align*}
\Vert z\Vert \leq \Vert w_1+w_2\Vert  \leq \Vert w_1\Vert +\Vert w_2\Vert\leq \rho(x)+\rho(y).
\end{align*}
Consequently, we have $\sup\{\Vert z\Vert:~{z\leq \vert x+y\vert}~\text{and}~z\in E^+\}
\leq \rho(x)+\rho(y)$, which implies that $\rho(x+y)\leq \rho(x)+\rho(y)$.
\end{proof}   
   
\vspace{0.1cm}   
   
For a normed lattice $(E,\Vert.\Vert)$, assume that  $E^\rho$  is including members of $x\in G$ which satisfies in the following equality,
\begin{equation}
\inf \{\Vert y\Vert:~\vert x\vert\leq y,~y\in E^+\}=\sup\{\Vert z\Vert:~z\leq \vert x\vert ,~z\in E^+\}
\end{equation}
 then  $\rho$ is defined  a  real function from  $E^\rho$ into $[0,+\infty)$, in the sense $\rho(x)$ equal to $(1)$. Obviously that the function $\rho$ satisfies in the following conditions:
\begin{enumerate}
\item $\rho(x)=0$ iff $x=0$
\item $\rho(\lambda x)=\lambda\rho(x)$ for each $\lambda\in \mathbb{R}$ and $x\in E^\rho$.
\item $\rho(x+y)\leq \rho (x)+\rho(y)$, whenever $x+y\in E^\rho$ for  $x,y\in E^\rho$.
\end{enumerate}

\begin{example}\label{E:2.2}
Let $c$ be the collection of all  real number sequences which convergence in $\mathbb{R}$ with $\ell^\infty$-norm. It is obvious that $c$ is order dense  majorizing of $\ell^\infty$. By easy calculation, we can prove that $c^\rho =\ell^\infty$. 
\end{example}

The above example make motivation to us for the following question.
\begin{question}
Is $E^\rho$ vector lattice and $(E^\rho, \rho)$ normed lattice?
\end{question}

It is important to us to know that when $E^\rho$ is a vector lattice and moreover $(E^\rho, \rho)$ is a normed lattice?   Now in the following we define a vector lattice $E^t\subseteq G$ which is including normed lattice $E$ with extension norm of $E$.

\begin{definition}
Assume that  $E\subseteq E^t$  is a   vector sublattice of $G$   in which every element of $E^t$ satisfies in the  equality $(1)$. 
Then  $\rho (x)=\Vert x\Vert_t$ is defined a norm for $E^t$ in the sense equal to $(1)$ which is called $t$-norm for $E^t$.
\end{definition}

It is easy to show that $\Vert .\Vert_t$ is a norm for $E^t$, and moreover $(E^t, \Vert .\Vert_t)$ is a normed lattice.  Note that $E^t$ is not unique and we have $E\subseteq E^t\subseteq G$. The aim of this manuscript to find a vector lattice $E^t$ in the  different from of $E$. So, when we say that $E^t$ exists, that is, $E$ is proper sublattice of $E^t$. Now in the Theorems  \ref{t2.1} and \ref{T:2.5}  we show that $E^t=G$ whenever $E$ is a Dedekind complete or has order continuous norm, respectively.

\begin{theorem}\label{T:2.5}
 By one of the following conditions, the equality $(1)$ holds for each $x\in G$, that is,  $E^t=G$, $(G,\Vert .\Vert_t)$ is normed lattice and 
  $\Vert y\Vert=\Vert y\Vert_t$ for each $y\in E$.
\begin{enumerate}
\item  $E$ is a Dedekind complete.
\item   $E$ has order continuous norm.
\end{enumerate}
\end{theorem}
\begin{proof}
\begin{enumerate}
\item By Theorem \ref{t2.1}, we know that  $\Vert x\Vert_t=\sup\{\Vert z\Vert:~z\leq \vert x\vert ,~z\in E^+\}$ is a norm for vector lattice $G$. 
By contradiction, assume that 
\begin{align*}
\Vert x\Vert_t\neq\inf \{\Vert y\Vert:~\vert x\vert\leq y,~y\in E^+\}.
\end{align*}
 Consider $\lambda$ is a real number such that 
\begin{align*}
\Vert x\Vert_t<\lambda <\inf \{\Vert y\Vert:~\vert x\vert\leq y,~y\in E^+\}.
\end{align*}
Let $A=\{y\in E^+:~\vert x\vert\leq y\}$. Since $E$ is order dense in $G$, $A$ is bounded below, and so $A$ has infimum in $E$, by Dedekind completeness of $E$. Take $\inf A=y_0$ where $y_0\in E$. It is clear that $\vert x\vert<y_0$ and $\lambda <\Vert y_0\Vert$. Let the natural number $n$ be enough large such that  $\lambda <\Vert y_0\Vert-\frac{1}{n}\Vert y_0\Vert$.
Put $z_0=(1-\frac{1}{n})y_0$. Consequently we have  $z_0\in A$ and $\Vert z_0\Vert<\Vert y_0\Vert$, which is impossible. 
\item First we show that 
\begin{equation*}
\inf \{\Vert y\Vert:~\vert x\vert\leq y,~y\in E\}=\sup\{\Vert z\Vert:~z\leq \vert x\vert,~z\in E\},
\end{equation*}
holds whenever $x\in G$. Set $A=\{z\leq \vert x\vert:~z\in E^+\}$ and $B=\{y\geq \vert x\vert:~y\in E\}$. Since $E$ is order dense and majorizing of  $G$, it follows that $A$ and $B$ are not empty and they are directed sets. We consider the set $A$ as a net $\{z_\alpha\}$, where $z_\alpha=\alpha$ for each $\alpha\in A$. In the same way we consider $B=\{y_\beta\}$, and by using Theorem 1.34 \cite{2}, we can write  
$z_\alpha\uparrow \vert x\vert$ and $y_\beta \downarrow\vert x\vert$. Since 
$z_\alpha\leq \vert x\vert\leq y_\beta$ for each $\alpha$ and $\beta$, it follows that $y_\beta -z_\alpha\downarrow 0$, and so  $0\leq\Vert  y_\beta\Vert -\Vert z_\alpha\Vert\leq\Vert  y_\beta -z_\alpha\Vert\rightarrow 0$. It follows that 
$\Vert x\Vert_t=\inf \Vert  y_\beta\Vert=\sup\Vert z_\alpha\Vert$. Obvious that $\Vert .\Vert_t$ is a norm for $G$ and $(G, \Vert .\Vert_t)$ is a normed lattice.
\end{enumerate}
\end{proof} 

\vspace{0.2cm}

In the Example \ref{E:2.2}, we see that $c^t=\ell^\infty$, but $c$ is not Dedekind complete and has not continuous norm. On the other hand, Theorem  \ref{T:2.5}, make a motivation to us that  we can extend the norm of $E$ for vector lattice $E^t$ in different cases. On the other hand, it is important to know that which time we have $({E^t})^t=E^t$. In the following example, we show that $E^t$ exists whenever $E$ satisfies in Fatou property. Note that by Example 4.3 and 4.4 from \cite{1}, we know that every normed lattice with Fatou property, in general sense has not order continuous norm or Dedekind complete. 
\begin{example}
By Theorem 4.12 \cite{1}, if $(E,\Vert.\Vert)$ satisfies the Fatou property, the Dedekind completion of $E$, $E^\delta$ is a normed space with $\delta-norm$. Let $E$ be the vector lattice of all real-valued functions defined on an infinite set $X$ whose range is finite, with the pointwise ordering and satisfies the Fatou property. It can be seen that $E$ is not Dedekind complete   and $E^\delta=\ell^\infty(X)$.
\end{example}

  Now in the following we introduce an important lemma that has many applications in every parts of this paper.
\begin{lemma} \label{2.1}
Let $E$ has order continuous norm.
For each $0\leq x\in E^t$, there are sequences $\{x_n\}\subseteq E^+$ and $\{y_n\}\subseteq E^+$ such that $x_n\uparrow x$,  $x_n\xrightarrow{\Vert \Vert_t}x$, $y_n\downarrow x$ and  $y_n\xrightarrow{\Vert \Vert_t}y$.
\end{lemma}
\begin{proof}
 Choose $\{r_n\}\subseteq \mathbb{R}^+$ and $\{x_n\}\subseteq E^+$ satisfies in the following conditions:
\begin{enumerate}
\item $r_n\downarrow 0$,
\item $x_n\in  \{z\in E:~z\leq x~\text{and}~\Vert x-z\Vert_t<r_n\}$,
\item $ x_n\uparrow x$,
\end{enumerate}
for each $n\in \mathbb{N}$. The reason  of the above claim as follows:\\
By using Theorem 1.34 \cite{2},  take $A=\{z\leq  x:~z\in E^+\}=\{z_\alpha\}$ and $B=\{y\geq  x:~y\in E\}=\{y_\beta\}$ such that 
$z_\alpha\uparrow  x$ and $y_\beta \downarrow x$. 
Then $z_\alpha \leq x \leq y_\beta$ holds for each $\alpha$ and $\beta$. Thus $\Vert x-y_\beta\Vert$, $\Vert x-z_\alpha\Vert\leq \Vert z_\alpha -y_\beta\Vert\rightarrow 0$. Let $0<r_1\in\mathbb{R}$.  Then there exist  $z_1\in\{z\in A:~\Vert x-z\Vert_t\leq r_1\}$ and $0<r_2<Min\{r_1,~\Vert z_1-x\Vert_t\}$. We choose $z_2,z_3,...z_n$ and $z_{n+1}\in\{z\in A:~\Vert x-z_{n}\Vert_t\leq r_n\}$ where $0<r_n<Min\{r_{n-1},~ \Vert z_{n-1}-x\Vert_t\}$. We define $x_n=\vee_{i=1}^n z_i$.
Now, if $x_n\leqslant w\leqslant x$ for each $n\in \mathbb{N}$, then $0\leqslant x-w\leqslant x-x_n\leq x-z_n$. It follows that 
$\Vert x-w\Vert_t\leqslant \Vert x-x_n\Vert_t\leqslant \Vert x-z_n\Vert_t \leqslant r_n\downarrow 0$. Thus $x=w$, and so $\sup x_n=x$. Therefore $x_n\uparrow x$ and $\Vert x_n-x\Vert\rightarrow 0$. Existence of $\{y_n\}$ has the same argument.
\end{proof}

\begin{theorem}\label{t2.3}
Let $E$ has order continuous norm.
Then we have the following assertions.
\begin{enumerate}
\item If $E$ is a $KB$-pace, then so is $E^t$.
\item If $E$ is an $AL-$space, then so is $E^t$.
\end{enumerate}
\end{theorem}
\begin{proof}
\begin{enumerate}
\item Assume that $\{x_n\}\subseteq (E^t)^+$ is increasing sequence which $\sup\Vert x_n\Vert_t<+\infty$. By using Lemma \ref{2.1}, for each $n\in  \mathbb{N}$, there are  increasing and positive sequences $\{x_{n,m}\}_m$ such that $x_{n,m}\uparrow_m x_n$ and $\Vert x_n- x_{n,m}\Vert_t\xrightarrow{m} 0$. Take $y_n=\bigvee_{i,j=1}^n x_{i,j}$. It follows that $0\leq y_n \uparrow$ and $\sup\Vert y_n\Vert\leq\sup_{i,j}\Vert x_{i,j}\Vert\leq \sup\Vert x_n\Vert<+\infty$. Since $E$ is a $KB$-pace, it follows that there exists $x\in E$ such that $\Vert y_n-x\Vert_t\rightarrow 0$. On the other hand,  the inequalities
$ y_n\leq  x_n\leq x$ implies that  $\Vert x_n-x\Vert_t\leq\Vert y_n- x\Vert_t$ for each $n\in  \mathbb{N}$. It follows that
$\Vert x_n-x\Vert_t\rightarrow 0$ holds in $E^t$.
\item Since $E$ is an $AL-$space, then $E$ has order continuous norm. Now, let $0<x,y\in E^t$ with $x\wedge y=0$. By using Lemma \ref{2.1}, there are $\{x_n\}$ and $\{y_n\}$ in $E^+$ such that  $x_n\uparrow x$,  $y_n\uparrow y$, 
$\Vert x-x_n\Vert_t\rightarrow 0$ and $\Vert y-y_n\Vert_t\rightarrow 0$. It follows that $0\leqslant x_n\wedge y_n\uparrow x\wedge y=0$ implies that $x_n\wedge y_n=0$ for each $n\in  \mathbb{N}$. Hence $\Vert  x_n+y_n\Vert=\Vert x\Vert+\Vert y_n\Vert$ for each $n\in \mathbb{N}$. Then $\Vert x+y\Vert_t=\lim_n \Vert x_n+y_n\Vert=\lim_n \Vert x_n\Vert+\lim_n\Vert y_n\Vert=\Vert x\Vert_t+\Vert y\Vert_t$. Consequently,  $E^t$ is an $AL-$space.
\end{enumerate}
\end{proof}

\begin{theorem}\label{2.4}
For a normed lattice $E$ with order continuous norm, we have the following assertions
\begin{enumerate}
\item If $\hat{E}$ is a norm completion of $E$, then $E^t\subseteq \hat{E}=E^u$, and if $E$ is norm complete, then $E^t=E^u=E$.
\item For each $x\in E^t$ and $A\subseteq E$ with $\sup A=x$, we have  $\Vert x\Vert_t=\sup_{z\in A}\Vert z\Vert$.
\item For each $x\in E^t$ and $A\subseteq E$ with $\inf A=x$, we have  $\Vert x\Vert_t=\inf_{z\in A}\Vert z\Vert$.
\item $(E^t,\Vert.\Vert_t)$ has Fatou property and $B_{E^t}=\{x\in E^t:~\Vert x\Vert_t\leq 1\}$ is order closed.
\item If $E$ is an ideal in $E^t$, then $\hat{E}=E^t$.
\end{enumerate}
\end{theorem}
\begin{proof}
\begin{enumerate}
\item By Theorem 2.40 \cite{1}, $(\hat{E}, \hat{\Vert .\Vert})$ is a normed lattice where $\hat{\Vert .\Vert}$ is unique extension of norm from $E$ into $\hat{E}$.
Let $x\in E^t$. Then by Lemma \ref{2.1}, there exists $\{x_n\}$ in $E^+$ such that  $x_n\uparrow x^+$ and
$\Vert x^+-x_n\Vert_t\rightarrow 0$. Thus  $\{x_n\}$  is a norm Cauchy sequence  in $E$, and so convergence in $\hat{E}$. It follows that $x^+\in \hat{E}$. In the similar way $x^-\in \hat{E}$, which implies that  $x\in \hat{E}$. Now by  Theorem 7.51 of \cite{1}, we conclude that $E^t\subseteq \hat{E}=E^u$ and $ \Vert .\Vert_t=\hat{\Vert .\Vert}$. On the other hand  if $E$ is norm complete, it is obvious that  $E^t=E^u=E$ and $\Vert .\Vert =\Vert .\Vert_t=\hat{\Vert .\Vert}$.
\item By using Theorem 7.54 \cite{1}, $E^u$ has order continuous norm. since by part (1), we have $E^t\subseteq E^u$, it follows that $E^t$ has order continuous norm.  Consider $A=(x_\alpha )$ with $\sup A=x$. It follows that $x-x_\alpha\downarrow 0$ which implies that $\Vert  x-x_\alpha\Vert_t\rightarrow 0$. Then by using inequalities
$0\leq\Vert x\Vert_t-\Vert x_\alpha\Vert\leq\Vert x-x_\alpha\Vert_t$, we have $\sup_\alpha \Vert x_\alpha\Vert=\Vert  x\Vert_t$. 
\item Proof has the same argument such as (2).
\item By using Lemma 4.2 \cite{1},  $(E,\Vert.\Vert)$ has Fatou property. The proof of first statement has similar argument such as  Theorem \ref{t2.3}(1),  we omit the proof. The second part, follows by Theorem 4.6 \cite{1}.
\item Proof follows by Theorem 3.8 \cite{1}.
\end{enumerate}
\end{proof}

\section{\bf Extension of order bounded operators}
In every parts of this section,  $T$ is order bounded operator from normed lattice $E$ into Dedekind complete normed lattice $F$.

\begin{theorem}\label{2.3}
We have the following assertions.
\begin{enumerate}
\item $T$ has an order bounded extension  $T^t$ from $E^t$ into $F$ such that $T^t(y)=Ty$ for each $y\in E$. 
\item  For each positive continuous operator $T$, we have $\Vert T\Vert=\Vert T^t\Vert$, and if $T$ is norm continuous, then so is $T^t$.
\item  $\vert T\vert^t=\vert T^t\vert$.
\item For each $T, S\in \mathcal{L}_b(E,F)$, we have $(T\vee S)^t=T^t\vee S^t$.
\item If $S:E^t\rightarrow F$ is an order bounded and norm continuous operator, then there exists order bounded and norm continuous operator $T:E\rightarrow F$ such that $T^t=S$.
\item Each order interval of $E^t$ is $\sigma (E^t, (E^t)^\prime )$-compact.
\end{enumerate}
\end{theorem}
\begin{proof}
\begin{enumerate}
\item Since $T$ is an order bounded operator and $F$ Dedekind complete, we can write  $T=T^+-T^-$. So first we assume that $T$ is a positive operator from $E$ into $F$. By notes the proof of Theorem 1.32 \cite{2}, the mapping $p:E^t\rightarrow F$  defined via the formula
 \begin{align*}
p(x)=\inf\{Ty:~y\in E, ~x\leqslant y\},~ x\in E^t.
\end{align*}
 is a monotone sublinear  and $Ty\leq p(y)$ for each $y\in E$. So by using Theorem 1.5.7 \cite{4}, there is an  extension of $T$,    $T^t$,  from $E^t$ into $F$ satisfying $T^tx\leq p(x^+)$ for all $x\in E^t$, and $T^ty=Ty$ for all $y\in E$. 
Now we define $T^t=(T^+)^t-(T^-)^t$ in which $T^t$ is an extension  of $T$ from $E^t$ into $F$ and for all $y\in E$, we have
\begin{align*}
T^t y=(T^+)^t (y)-(T^-)^t (y)=T^+y-T^- y=Ty.
\end{align*} 
\item First assume that $T$ is a positive operators and $x\in E^t$.  By part (1),  we have  $T^tx\leq p(x^+)\leq Ty$ for all $ y\geq x^+$, and so  $\Vert T^tx\Vert\leq  \Vert Ty\Vert$ for all $y\geq x^+$ which implies that   $\Vert T^tx\Vert\leq \Vert T\Vert\inf_{y\geq x^+}\Vert y\Vert \leq \Vert T\Vert\Vert x^+\Vert_t\leq \Vert T\Vert\Vert x\Vert_t$. Then  $\Vert T^t\Vert\leq \Vert T\Vert$.
Since $B_E\subseteq B_{E^t}$, follows that $\Vert T\Vert\leqslant    \Vert T^t \Vert$. Thus  $\Vert T\Vert=    \Vert T^t \Vert$.
Thus $(T^-)^t$ and  $(T^+)^t$ are norm continuous. This follows that $T^t$ is a norm continuous operator from $E^t$ into $F$. 
\item In this part, we assume that $x,y,z$ are members of $E$ and $x^t,y^t,z^t$ are members of $E^t$ when there is not any confused. Now let $x^t\geq 0$. Since $E$ is order dense in $E^t$, we have the following equalities
\begin{align*}
{(T^t)}^+(x^t)&=\sup_{0\leq y^t\leq x^t}T^ty^t\\
&=\sup_{0\leq y^t\leq x^t}\sup_{0\leq z\leq y^t}T^tz\\
&=\sup_{0\leq y\leq x^t}Ty\\
&=\sup_{0\leq z\leq x^t}\sup_{0\leq y\leq z}Ty\\
&=\sup_{0\leq z\leq x^t}T^+z\\
&=(T^+)^t (x^t).
\end{align*}
In the same way ${(T^t)}^-(x^t)=(T^-)^t (x^t)$ for all $x^t\geq 0$.
It is obvious that for each $x^t\in E^t$, we have ${(T^t)}^+x^t=(T^+)^t (x^t)$ and ${(T^t)}^-x^t=(T^-)^t (x^t)$. 
Thus 
\begin{align*}
\vert T\vert^t=(T^++T^-)^t=(T^+)^t+(T^-)^t=(T^t)^++(T^t)^-=\vert T^t\vert.
\end{align*}
\item By using equality $T\vee S=\frac{1}{2}(T+S+\vert T-S \vert)$ and part (3), proof follows.
\item First let $0\leq x\in E^t$. By Lemma \ref{2.1}, there exists $\{x_n\}$ in $E^+$ such that  $x_n\uparrow x^+$ and
$\Vert x^+-x_n\Vert_t\rightarrow 0$. Since $S^+x_n\uparrow$ and $\Vert x^+-x_n\Vert\rightarrow 0$, follows that $S^+x_n\uparrow S^+x$. We define $T={S\vert}_E$ (restriction  of $S$ on $E$), which follows that $T^-={S^-}\vert_E$  and $T^+=S^+\vert_E$ . Obviously   $(T^-)^t=S^-$ and  $(T^+)^t=S^+$, and so by part (3), we have the following equalities
\begin{align*}
S=S^+-S^-=(T^+)^t-(T^-)^t=(T^t)^+-(T^t)^-=T^t.
\end{align*}
Thus $S=T^t$ on $E^-$ and $E^+$, which follows that
\begin{align*}
Sx=Sx^+-Sx^-=T^tx^+-T^tx^-=T^tx, 
\end{align*}
 for each $x\in E^t$.
\item Consider $a,b\in (E^t)^+$ and $a< b$. By Lemma \ref{2.1}, take $\{x_n\}$  and  $\{y_n\}$ in $E^+$ such that  $x_n\uparrow a$, $y_n\downarrow b$,
$\Vert a-x_n\Vert_t\rightarrow 0$ and $\Vert y_n-b\Vert_t\rightarrow 0$. Since $E$ has order continuous norm, $[x_n,y_n]\cap E$ is $\sigma (E,E^\prime )$-compact subset of $E$ for each $n\in  \mathbb{N}$. It follows that $[a,b]\cap E$ is $\sigma (E,E^\prime )$-compact subset of $E$.  Now, if $V=\{s\in E:~x^\prime (s)<r ~\text{and}~x^\prime\in E^\prime\}$, then by using part (5),  the order density of $V$ is $V^t=\{s\in E^t:~(x^\prime)^t (s)<r ~\text{and}~(x^\prime)^t\in (E^t)^\prime\}$. Thus obvious that $V\subseteq V^t$, and so  $\sigma (E, E^\prime )\subseteq  \sigma (E^t, (E^t)^\prime )$. Since  $[a,b]\cap E$ is order dense in $[a,b]$, follows that  $[a,b]$ is  $\sigma (E^t, (E^t)^\prime )$-compact subset of $E^t$.
\end{enumerate}
\end{proof}

\vspace{0.3cm}

Now in the following, we investigate on some properties of operator $T^t$ from $E^t$ into $F$. We show that $T^t$ is keeping some lattice and topological properties whenever $T$ was done.

\begin{theorem}
Let   $0\leq T\in \mathcal{L}_n(E,F)$. Then we have the following assertions
\begin{enumerate}
\item If $0\leq x\leq E^t$ and $\{x_\alpha\}\subseteq E^+$ with  $x_\alpha\downarrow x$ , then $Tx_\alpha\downarrow T^t x$.
\item   If $T(x\wedge y)=Tx \wedge Ty$ for each $0\leq x,y\in E$, then $T^t$ is a lattice homomorphism from $E^t$ into $F$  and  moreover   $T^t\in \mathcal{L}_n(E^t,F)$.
\item If $0\leq T:E\rightarrow E$ is a band preserving, then $T^t:E^t\rightarrow E^t$ is so.
\item If $T:E\rightarrow F$ is preserving disjointness, then $T^t$ is so. 
\item Let $E$ has order continuous norm.
 $\{Tx_n\}$ is norm convergent in $F$ for each positive increasing norm bounded sequence $\{x_n\}$ in $E$ iff $\{T^tx_n\}$ is norm convergent in $F$ for each positive increasing $t$-norm bounded sequence $\{x_n\}$ in $E^t$.
\end{enumerate}
\end{theorem}
\begin{proof}
\begin{enumerate}
\item Let $\{x_\alpha\}\subseteq E^+$ such that $x_\alpha\downarrow x$.
If $y\in E^+$ such that $x\leq y$, then $y\vee x_\alpha\downarrow y$ holds in $E$, and so by  order continuouty of $T:E\rightarrow F$ and Theorem \ref{2.4} (3),  we see that
\begin{align*}
Ty=\inf\{T(x_\alpha \vee y)\}\leq \inf Tx_\alpha \leq T^tx.
\end{align*}
This easily implies that $Tx_\alpha \downarrow T^t x$. \\
\item Assume that  $0\leq x,y\in E^t$.  We prove that $T^t(x\wedge y)=T^tx\wedge T^ty$. By Theorem 1.34 \cite{2}, there are $\{x_\alpha\}$ and $\{y_\beta\}$ of $E^+$ such that $x_\alpha\downarrow x$ and $y_\beta\downarrow y$. It follows that $x_\alpha\wedge y_\beta \downarrow x\wedge y$. Then by  order continuouty of $T:E\rightarrow F$ and Theorem \ref{2.4} (3), we have the following equalities, 
\begin{align*}
T^t(x\wedge y)&= \inf \{T(x_\alpha\wedge y_\beta )\}=\inf \{T(x_\alpha) \wedge T(y_\beta )\}\\
&=\inf \{T(x_\alpha)\} \wedge \inf \{T(y_\beta )\}=T^tx\wedge T^ty.
\end{align*}
Combining Theorem 1.10,  Theorem 2.14 from \cite{2} and Theorem \ref{2.3}, we see that the mapping $T^t:(E^t)^+\rightarrow (F^t)^+$ extends to  $T^t:E^t\rightarrow F^t$ which is lattice homomorphism.\\
Now, we show that  $T^t\in \mathcal{L}_n(E^t,F)$.  Let $\{x_\alpha\}\subseteq (E^t)^+$ such that $x_\alpha\downarrow 0$. Put 
\begin{align*}
A=\{y\in E^+:~\exists \alpha~ \text{such~that}~ x_\alpha \leq y\}.
\end{align*}
Since $E$  mejorizes $E^t$, it follows that $A$ is not empty.  By using Theorem \ref{2.3} since $T$ is positive, $T^t$ is positive.    Thus  $\inf T(A)\geq \inf T^t x_\alpha\geq 0$ holds in $F$. Since $A\downarrow 0$ and $T\in \mathcal{L}_n(E,F)$, it follows that $\inf T(A)=0$, and so $T^t x_\alpha\downarrow 0$.
\item  Let $x,y\in E^t$ satisfying  $ \vert x\vert\wedge  \vert y\vert=0$. Assume that $(x_\alpha), (y_\beta)\subseteq E^+$ such that $x_\alpha\uparrow \vert x\vert$ and $y_\beta\uparrow \vert y\vert$. It follows that $(x_\alpha\wedge y_\beta)\uparrow \vert x\vert \wedge  \vert y\vert=0$, and so  $x_\alpha\wedge y_\beta=0$, which by using Theorem 2.36 \cite{2}  implies that $\vert Tx_\alpha\vert\wedge y_\beta=0$ for each $\alpha$ and $\beta$.
Since $\vert Tx_\alpha\vert\wedge y_\beta\uparrow \vert Tx\vert \wedge \vert y\vert$, we have $\vert Tx\vert \perp \vert y\vert$,     and so by another using Theorem 2.36 \cite{2}, proof follows.
\item Let $x,y\in E^t$ satisfying $x\perp y$.  Assume that $(x_\alpha), (y_\beta)\subseteq E^+$ such that $x_\alpha\uparrow \vert x\vert$ and $y_\beta\uparrow \vert y\vert$. It follows that $(x_\alpha\wedge y_\beta)\uparrow \vert x\vert \wedge  \vert y\vert=0$. Now since $T$ preserve disjointness, follows that $Tx_\alpha\perp Tx_\beta$. From our hypothesis, we have $Tx_\alpha\wedge Tx_\beta\uparrow T\vert x\vert\wedge T\vert y\vert$ which follows that $T\vert x\vert\wedge T\vert y\vert=0$. Since $\vert Tx\vert\wedge \vert Ty\vert\leq T\vert x\vert\wedge T\vert y\vert$, we have $Tx\perp Ty$.
\item Since $T=T^+-T^-$, without loss generality, we assume that $T$ is a positive operator.
Assume that  $\{x_n\}\subseteq (E^t)^+$ is increasing sequence with $\sup\Vert x_n\Vert_t<+\infty$.  By using Lemma \ref{2.1}, for each $n\in  \mathbb{N}$, there are  positive increasing  sequences $\{x_{n,m}\}_m$ with $x_{n,m}\uparrow_m x_n$ and $\Vert x_n- x_{n,m}\Vert_t\rightarrow 0$. Take $y_n=\bigvee_{i,j=1}^n x_{i,j}$. It follows that $0\leq y_n \uparrow$ and $\sup\Vert y_n\Vert\leq\sup_{i,j}\Vert x_{i,j}\Vert\leq \sup\Vert x_n\Vert<+\infty$. By assumption there is $s^*\in F$ such that $\Vert Ty_n-s^*\Vert\rightarrow 0$. Then by  using Theorem 2.46 \cite{2}, $Ty_n\uparrow s^*$. By Theorem \ref{2.3}, we know that  $T^t$ is norm continuous from $E^t$ into $F$. It follows that $\Vert T^tx_n- Tx_{n,m}\Vert\xrightarrow{m} 0$ holds in $F$. The inequality $Tx_{n,m}\leq Ty_n\leq T^t x_n$ implies that 
\begin{align*}
\Vert T^tx_n-s^*\Vert\leq\Vert T^tx_n- Tx_{n,m}\Vert \; \text{for each} \;  n,m\in  \mathbb{N}.
\end{align*}
  Then 
\begin{align*}
 \Vert T^tx_n-s^*\Vert\leq\Vert T^tx_n-Ty_n\Vert+\Vert Ty_n-s^*\Vert\rightarrow 0.
\end{align*}
 Thus $T^tx_n\rightarrow s^*$, and  the proof follows.\\
 The converse is clear.
\end{enumerate}
\end{proof}

In the study of extension of normed lattice $E$ and linear operator $T$, we have some questions as follows. In the following questions,  $\mathcal{K}(E,F)$ and $\mathcal{KW}(E,F)$  are the collections of compact operators and weakly compact operators from $E$ into $F$, respectively.
\begin{question} 
\begin{enumerate}
\item Are $T^t\in \mathcal{K}(E^t,F)$ and   $T^t\in \mathcal{WK}(E^t,F)$ whenever  $ T\in \mathcal{K}(E,F)$ and $ T\in \mathcal{WK}(E,F)$, respectively?
\item  Has $T^t$ Dunford-Pettis operator when $T$ has?
\end{enumerate}
\end{question}
\begin{question}
When the equality $(T^t)^t=T^t$ holds.
\end{question}

%\begin{acknowledgements}
%The authors would like to thank the anonymous referee for his/her valuable comments.
%\end{acknowledgements}

%%%%%%%%%%%%%%%%%%%%%%%%%%%%%%%%%%%%%%%%%%%%%%%%%%%%%%%%%%%%%%%%%%%%%%%%%%%%%%%%%%%%%%%%%%%%%%%%%%%%%%%%%%%%%%%%%%%%%%%%%%%%%%%%%%%%%%%%%%%%%%%%%%%%%%%

\end{document}